\documentclass[11pt,oneside,a4paper,leqno]{article}
\usepackage[utf8]{inputenc}
\usepackage[english]{babel}

\usepackage{graphicx}
\usepackage{amsfonts,amssymb}
\usepackage[centertags]{amsmath}

\usepackage{amsthm} 
\newcounter{theorems}

\theoremstyle{plain}

\swapnumbers
\newcounter{lemma}
\numberwithin{equation}{section}

\newtheoremstyle{par}
     {\topsep}
     {\topsep}
     {\itshape}
     {}
     {\bfseries}
     {}
     {.5em}
     {}

\newtheoremstyle{parrm}
     {\topsep}
     {\topsep}
     {\normalfont}
     {}
     {\itshape}
     {}
     {.5em}
     {}

\theoremstyle{plain}
\numberwithin{equation}{section}
\newtheorem{theo}[equation]{Theorem}
\newtheorem{coro}[equation]{Corollary}

\theoremstyle{definition}
\newtheorem{defi}[equation]{Definition}
\newtheorem{example}[equation]{Example}

\theoremstyle{remark}
\newtheorem{remark}[equation]{Remark}

\theoremstyle{par}

\newtheorem{lemma}[equation]{}
\newtheorem{propo}[equation]{}

\theoremstyle{parrm}

\newcommand{\ind}{\operatorname{ind}}

\makeatletter
\def\tagform@#1{\maketag@@@{\ignorespaces#1\unskip\@@italiccorr}}

\makeatother

\usepackage{paralist}
\usepackage{comment}
\usepackage{subfigure}
\newcommand{\RR}{\mathbb{R}}
\newcommand{\CC}{\mathbb{C}}

\newcommand{\PP}{\mathbb{P}}

\newcommand{\from}{\colon}

\usepackage{bm}
\newcommand{\simbolovettore}[1]{{\boldsymbol{#1}}}

\newcommand{\vc}{\simbolovettore{c}}

\newcommand{\vq}{\simbolovettore{q}}

\newcommand{\vu}{\simbolovettore{u}}
\newcommand{\vv}{\simbolovettore{v}}
\newcommand{\vw}{\simbolovettore{w}}
\newcommand{\vx}{\simbolovettore{x}}

\newcommand{\vz}{\simbolovettore{z}}
\newcommand{\zero}{\boldsymbol{0}}
\newcommand{\abs}[1]{\lvert{#1}\rvert}
\newcommand{\Fix}{\operatorname{Fix}}

\usepackage{multirow,rotating}

\newcommand{\term}[1]{\emph{#1}}
\newcommand{\norm}[1]{\lVert{#1}\rVert}
\newcommand{\grad}{\nabla}
\begin{document}
\pagenumbering{arabic}

\title{%
Fixed point indices of central configurations
}

\author{D.L.~Ferrario}

\date{%
\today}

\maketitle
\centerline{\itshape To Professor Andrzej Granas}

\begin{abstract}
Central configurations of $n$ point particles in $E\approx \RR^d$ with 
respect to a potential function $U$ are shown to be the same as the 
fixed points of the normalized gradient map 
$F=-\nabla_M U / \lVert \nabla_M U \rVert_M$, 
which is an $SO(d)$-equivariant self-map defined on the 
inertia ellipsoid. We show that the $SO(d)$-orbits of fixed points of $F$ 
are all fixed points of the map induced on the quotient by $SO(d)$, 
and give a formula relating their indices (as fixed points) with 
their Morse indices (as critical points). 
At the end, we give an example of a non-planar relative equilibrium which is
not a central configuration.

\noindent {\em MSC Subject Class\/}:
55M20,
37C25,
70F10.


\end{abstract}

\section{Central configurations}
\label{sec:intro}

Let $E=\RR^d$ be the $d$-dimensional Euclidean space, and 
$X = E^n \smallsetminus \Delta$ the
\term{configuration space} of $E$, defined as the
space of all $n$-tuples of distinct
points $\vq =(\vq_1,\ldots, \vq_n)  \in E^n$
such that $\vq \not\in \Delta$,
where $\Delta $ is the collision set
\[
\Delta = \bigcup_{i<j} \{ \vq  \in E^n: \vq_i = \vq_j \}~.
\]

Let $U\from X \to \RR$ be a regular
potential function. 
For example, the gravitational
potential
\begin{equation}
\label{eq:newtonianU}
U(\vq) = \sum_{i<j} \dfrac{m_im_j}{\norm{\vq_i-\vq_j}^\alpha},
\end{equation}
or more general potentials (charged particles, \textellipsis). 

Given $n$ real non-zero numbers $m_1$, \ldots, $m_n$ (representing
the \term{masses} of the $n$ interacting point particles), 
let $\langle -,- \rangle_M$ denote the \term{mass-metric} on $E^n$,
defined as 
$\langle \vv,\vw\rangle_M = \sum_{j=1}^n m_j \vv_j \cdot \vw_j$,
where $\vv,\vw\in \RR^{dn}$ are two vectors tangent to $E^n$, 
and $\cdot$ denotes the standard scalar product in $E$.
If the masses are positive, then the mass-metric is 
a non-degenerate
scalar product on $E^n$,
which yields
both the kinetic quadratic form (on the tangent bundle)
$2 K =
\norm{\dot\vq}^2_M =
\sum_{i} m_i \norm{\dot \vq_i}^2$
and the inertia form $I= \norm{\vq}_M^2 = \sum_i m_i \norm{\vq_i}^2$.

Given such $m_i$, let $\grad_M$ denote the gradient of $U$ 
with respect to the bilinear product $\langle -,-\rangle_M$. 
Let us recall that if $dU(\vq)$ denotes the differential of $U$ evaluated in $\vq$, then 
for each tangent vector $\vv \in T_{\vq}E^n $, 
$dU(\vq)[\vv] = \langle \grad_M U(\vq), \vv \rangle_M$,
and therefore 
\[
dU(\vq)[\vv] 
= 
\sum_{j} \dfrac{\partial U}{\partial \vq_j} \cdot \vv_j  
= 
\sum_{j} m_j \left( m_j^{-1} \dfrac{\partial U}{\partial \vq_j} \right) \cdot \vv_j ,
\] 
from which it follows that in standard coordinates
$(\grad_M U)_j = m_j^{-1}  \dfrac{\partial U}{\partial \vq_j}$.

Given the mass-metric gradient $\grad_M U$, the corresponding 
\term{Newton equations} are 
\begin{equation}\label{eq:newton}
\dfrac{d^2\vq}{dt^2} =  \grad_M U (\vq) .
\end{equation}

\begin{defi}
A configuration $\vq \in X$ is a
\term{central configuration} iff there exists $\lambda \in \RR$, $\lambda \neq 0$, such that 
$\grad_MU(\vq) = \lambda \vq$. 
\end{defi}

\begin{remark}
If $U$ is homogeneous, this is equivalent to:  
$\vq$ is a central configuration iff 
there is a real-valued function $\varphi(t)$ such that 
$\varphi(t)\vq$ is a solution of \ref{eq:newton}\footnote{%
Cf. \cite{MR856309}, p.~61. 
}. 
If $U$ is homogeneous, the set of central configurations is a cone in $X$. 
\end{remark}

\begin{remark}
Furthermore, if $U$ is invariant with respect to the group of all 
translations of $E$, 
then central configurations belong to  the subspace
\[
Y = \{ \vq \in X : \sum_{j} m_j \vq_j = 0 \} \subset X \subset E^n,
\]
and 
$\forall \vq \in X \implies \grad_M U(\vq) \in \overline Y\subset E^n$,
where $\overline Y$ is the closure of $Y$ in $E^n$.   
Therefore, if $U$ is translation-invariant, the set of central configurations 
is a subset of $Y$. Sometimes central configurations 
are defined with the equation
$\nabla_MU(\vq) = \lambda ( \vq - \vc)$,
where $\vc$ is the center of mass $\vc =(\sum_jm_j)^{-1}\sum_{j} m_j \vq_j$
of the configuration $\vq$. This equation is invariant with respect to translations
(if $U$ is so). 
\end{remark}


\begin{defi}
A configuration $\vq\in X$ is a 
\term{relative equilibrium} iff 
there is a one-parameter group of rotations $\varphi^t\from E \to E$   
(around the origin, without loss of generality)
such that 
\[ \varphi^t(\vq) = (\varphi^t(\vq_1),\varphi^t(\vq_2), \ldots, \varphi^t(\vq_n) ) \]
satisfies 
the equations of motion \ref{eq:newton}.\footnote{%
Cf. \cite{MR0321138}, p.~47.
}
\end{defi}

One-parameter subgroups of $SO(N)$ are of the form 
$\varphi^t(\vq_1) = e^{t\Omega}\vq_1$, with 
$\Omega$ skew-symmetric $N\times N$ (non-zero) matrix.\footnote{Cf. \cite{MR532830}
, p. 401.}
If $U$ is invariant with respect to the above-mentioned one-parameter
 group of rotations $\varphi^t =e^{t\Omega}$,  
a relative equilibrium satisfies the equation
\[
\Omega^2 \vq = \nabla_M U (\vq)~.
\]
It follows that if $\dim(E) = 2$, then $\Omega= \begin{bmatrix}0&-\omega\\\omega&0\end{bmatrix}$
with $\omega\in \RR$, $\omega\neq 0$,
and  therefore such a relative equilibrium configuration
is a central configuration, $-\omega^2 \vq =  \nabla_M U(\vq)$.
 Conversely, planar
central configurations with $U(\vq)>0$ (if $U$ is homogeneous of negative degree)
yield 
relative equilibria, with a suitable (angular speed) $\omega$.  

If $\dim(E) = 3$, then since the non-zero $3\times 3$ skew-symmetric matrix $\Omega$
has rank $2$, $E$ can be written as $\ker \Omega \oplus E'$, where
$\ker \Omega$ is the fixed direction of the rotations $\varphi^t = e^{t\Omega}$,
and 
$E'$ is the orthogonal complement of $\ker \Omega$. In a suitable
reference, 
$\Omega = \begin{bmatrix}0& -\omega& 0 \\ \omega & 0 & 0 \\ 0 & 0 & 0 \end{bmatrix}$.
For further reference, let 
\begin{equation}
\label{eq:projectionP}
P\from E \to E'
\end{equation}
denote the orthogonal projection.

If $U$ is the homogeneous Newtonian potential
of \ref{eq:newtonianU}
with $\alpha>0$ and $m_j>0$, 
then 
it is easy to see that relative equilibrium configurations
must belong to the plane $E'$. 
This is not true in general: it is possible to find examples
of relative equilibria which are not planar -- see \ref{example:c1c2c3} (and hence
they are not central configurations). 
%
For more on equilibrium (and homographic) solutions: 
\cite{wintner} (§369--§382bis at pp.~284--306),
\cite{betti1877sopra},
\cite{ac98}.

Recent and non recent relevant literature on central configurations 
is the 
following:
\cite{MR2139425},
\cite{Mo90},
\cite{MR0309153,MR0321138},
\cite{MR0321389,MR0363076,MR0413647},
\cite{MR856309},
\cite{Xia},
\cite{MR1320359},
\cite{HaM2207019},
\cite{MR2925390}.


From now on, unless otherwise stated, assume that $U$ is invariant with respect 
to all isometries in $E$, all masses $m_j>0$ 
are positive, 
and $U$ is homogeneous of negative degree $-\alpha$. 

The potential $U$ is invariant with respect to a suitable subgroup 
of $\Sigma_n \times O(E)$, where $\Sigma_n$ is the 
symmetric group on $n$ elements and $O(E)$ denotes the 
orthogonal group on the euclidean space $E$. 
For example, if all masses are equal and $U$ is defined as in 
\eqref{eq:newtonianU}, then $G=\Sigma_n \times O(E)$; 
if all masses are distinct, then $G = \{1\} \times O(E)$. 

The following proposition is a well-known characterization of the set of 
central configurations, which we will
generalize to relative equilibria in \ref{lemma:c1c2c3}. 

\begin{lemma}
Let $S\subset Y $ denote the inertia ellipsoid,
defined as $S = \{ \vq  \in Y : \lVert \vq \rVert_M^2 = 1 \}$.  
A point $\vq\in S$ is a central configuration
if and only if it is  a critical point of the restriction 
of $U$ to $S$. 
\end{lemma}
\begin{proof}
Critical points of $U|_S$ are points $\vq\in Y$ such that 
$\ker dU \supset T_\vq S$. With respect to the (non-degenerate) bilinear
form $\langle -,-\rangle_M$, this can be written as 
$\nabla_M U(\vq) = \lambda \nabla_M (\norm{\vq}_M^2) = 2 \lambda \vq$.
\end{proof}

\begin{lemma}
\label{lemma:c1c2c3}
Assume $\dim(E) = 3$. 
Let $C$ be the \emph{vertical cylinder} defined as 
\[
C = \{ \vq \in Y : \langle P \vq , \vq \rangle_M =  c~,
\}
\]
where $P$ is the projection of $E$ to $E'$ as in \ref{eq:projectionP}
and $c= \langle P \overline \vq, \overline \vq \rangle_M$. 
A configuration $\overline\vq \in Y$ is a relative equilibrium configuration 
rotating by $e^{t\Omega}$ if 
and only if it is a critical point of $U$ restricted to 
$C\subset X$ and $U(\overline\vq)>0$.
\end{lemma}
\begin{proof}
The configuration $\vq$ is a relative equilibrium configuration 
if and only if 
$\Omega^2 \vq = \nabla_M U(\vq)$; since 
$\Omega^2 = -\omega^2 P$ for an $\omega\neq 0$,  
this is equivalent to 
\[
\nabla_M U(\vq) = - \omega^2 P\vq. 
\]

On the other hand, $\nabla_M \left( \langle P\vq, \vq \rangle_M \right) 
= 2 P\vq$,
hence $\vq \in C$ is a critical point of the restriction $U|_C $ iff 
$\nabla_M U(\vq) = 2 \lambda  P \vq$.  
The proof follows since, by homogeneity,
\[
\langle \nabla_M U(\vq) , \vq \rangle_M = -\alpha U(\vq)
\]
and $\langle P\vq, \vq \rangle_M = \norm{P\vq}^2_M$.
\end{proof}

\begin{lemma}
\label{lemma:equivCC}
Let $K$ be a subgroup of the symmetry group $G$ of $U$ on $Y$. 
Then the inertia ellipsoid $S$ is $K$-invariant, 
and critical points of 
the restriction of $U$ to $S^K= \{ \vq \in S : K\vq = \vq\}$ 
are precisely the critical points of $U|_S$ belonging to $S^K$. 
If the vertical cylinder $C$ is $K$-invariant, then critical points
of the restriction of $U$ to $C^K = \{ \vq \in C : K\vq=\vq \}$ are  
precisely the critical points of $U|_C$ belonging to $C^K$. 
\end{lemma}
\begin{proof}
The 
Palais principle of Symmetric Criticality
\cite{palais} states that \emph{in order for a symmetric point $p$
to be a critical point it suffices that it be a critical point of 
$f|_\Sigma$, the restriction of $f$ to $\Sigma$},
where $f\from M \to \RR$ is a smooth $K$-invariant function on $M$,
$M$ is a smoorh manifold and a symmetric point is an element 
of $\Sigma = M^K = \{ p\in M : gp=p \text{ for all } g\in K \}$. 
Thus, the statement follows by setting $f=U|_M$ with $M=S$ or $M=C$.
\end{proof}

\section{Central configurations as (equivariant) fixed points.}

In \cite{MR2372989} a way 
to relate planar central configurations to projective classes of fixed points 
was introduced.
We now generalize the results therein to arbitrary dimensions. 

Consider a homogeneous potential $U$, as above  with the further 
assumption that $\forall \vq, U(\vq)>0$. From this it follows that 
$\nabla_MU(\vq) \neq \zero$  because
 $\langle \nabla_M U(\vq), \vq \rangle_M = -\alpha U(\vq)$.

\begin{propo}
\label{propo:mapFStoS}
The map
$F\from S \to S$ defined as 
\[
F(\vq) = - \dfrac{ \grad_M U(\vq)}{\norm{\grad_M U(\vq) }_M }
\]
is well-defined, and 
a configuration $\bar\vq \in  S$ is a central configuration if and only if 
it is a fixed point of $F$. 
\end{propo}
\begin{proof}
It follows from the assumption that  
$\forall \vq, \nabla_MU(\vq) \neq \zero$,
and therefore $F$ is well-defined. 
A configuration $\vq$ is fixed by $F$ if and only if there 
exists $\lambda = \norm{\grad_M U(\vq)}_M >0$ such that 
$\nabla_M U(\vq) = - \lambda \vq$. Hence, if $F(\vq) = \vq$ 
then $\vq$ is central. 
 Conversely, 
by homogeneity of $U$,  
$\langle \nabla_M U(\vq), \vq \rangle_M = -\alpha U(\vq) $,
and hence if $\vq\in S$ is a central configuration 
then $\nabla_M U(\vq) = \lambda \vq $ $\implies $ 
$\lambda = -\alpha U(\vq)<0$ and therefore $F(\vq) = \vq$. 
\end{proof}

Now, 
in general (and without the positivity assumption)
the map $F$ needs not being compactly fixed (see for example 
Robert's continuum \cite{MR1669486} of central configurations 
with four unit masses and a fifth negative $-1/4$ mass in the origin\footnote{%
In the AMS review of \cite{MR1669486}, D. Saari
states that {``a similar effect occurs for positive masses, 
but with two or more homogeneous potentials''}. 
More precisely (in review MR1695029), ``{should the potential be a sum of homogeneous potentials, then there always is a continuum of different relative equilibria configurations}''. 
In review MR2245515:
``there can be a continuum of different central configurations parametrized by the value of the moment of inertia''.
}). 
In the gravitational case (positive masses and Newtonian mutual
attraction), the map $F$ turns out to be compactly fixed \cite{MR2372989} (see 
also Shub's estimates \cite{MR0278700}). 


\begin{propo}
\label{propo:inducedF}
Let $G$  be the symmetry group of $U$ on $Y$, as above.  
Then $F$ is $G$-equivariant (that is, for each $g\in G$ and for each $\vq\in S$, 
$F(g\vq) = gF(\vq)$). For each subgroup $K\subset G$, 
$F$ induces a self-map $\bar F \from S/K  \to S/K$ 
on the quotient space $S/K$.
\end{propo}
\begin{proof}
For $g\in G$, $gS=S$, and for each $g$ such that 
$U(g\vq) = U(\vq)$ the equality $\nabla_MU(g\vq) = g \nabla_M U(\vq)$ 
holds. 
In fact, 
since $U\circ g = U$, $dU \circ g = dU$, and therefore
for each vector $\vv$ one has
$
\langle \nabla_M U(\vq), \vv \rangle_M  = 
dU(\vq)[\vv] = dU(g\vq) [g\vv] = 
\langle \nabla_MU(g\vq),g\vv\rangle_M = 
\langle g^{-1} \nabla_M U(g\vq), \vv\rangle_M
$.
Thus $F$ is $G$-equi\-va\-ri\-ant, and hence $K$-equi\-va\-ri\-ant for each 
subgroup $K\subset G$. 
\end{proof}

Let $U$ be a homogeneous potential with the following property: 
for each orthogonal projection $p \from E \to \eta$ on a plane $\eta$, 
for each $\vq \in S$ there exists $j \in \{ 1, \ldots, n \}$ 
such that 
\begin{equation} 
\label{eq:propertyF}
p( \dfrac{\partial U}{\partial \vq_j} (\vq) ) \cdot p ( \vq_j) \leq 0~.
\end{equation}
Moreover, if there exists $i\in \{1,\ldots, n\}$ such that $p(\vq_i) \neq 0$, 
then the $j$ of \eqref{eq:propertyF} is such that 
$p(\vq_j)\neq 0$. 

It is easy to see that the Newtonian potential \eqref{eq:newtonianU}
(with positive masses and homogeneity $-\alpha$)
satisfies 
Property \eqref{eq:propertyF}:
let $j$ be the index maximizing 
$\norm{p(\vq_k)}^2$ for $k=1,\ldots, n$; then 
\[
\begin{aligned}
p \left( 
\alpha \sum_{k\neq j} \dfrac{m_j m_k(\vq_k - \vq_j)}{\norm{\vq_k - \vq_j}^{\alpha+2} }
\right) \cdot p(\vq_j)   & = 
\alpha \left( 
 \sum_{k\neq j} \dfrac{m_j m_k(p(\vq_k) - p(\vq_j))}{\norm{\vq_k - \vq_j}^{\alpha+2} }
\right) \cdot p(\vq_j) \\
&= 
\alpha 
 \sum_{k\neq j} \dfrac{m_j m_k(p(\vq_k)\cdot p(\vq_j) - \norm{p(\vq_j)}^2)}{\norm{\vq_k - \vq_j}^{\alpha+2} }
\\
& \leq \alpha 
 \sum_{k\neq j} \dfrac{m_j m_k(\norm{p(\vq_k)}\norm{ p(\vq_j)} - \norm{p(\vq_j)}^2)}{\norm{\vq_k - \vq_j}^{\alpha+2} }
\\
& \leq 0
\end{aligned} 
\]
It is trivial to see that 
if $U$ satisfies \eqref{eq:propertyF}, then the map $F\from S \to S$ defined in 
\ref{propo:mapFStoS} satisfies the following property:
for each orthogonal projection $p\from E \to P$ on a plane $P$, 
for each $\vq\in S$,  there exists $j\in \{1,\ldots, n\}$ such that 
\begin{equation}
\label{eq:propertyFF}
p( F_j(\vq) ) \cdot p(\vq_j) \geq 0~.
\end{equation}

\begin{theo}
\label{theo:FF} 
Suppose that $G\supset SO(E)$, where as above $G$ is the symmetry group of $U$ 
on $Y$, and $SO(E)=SO(d)$ (the group of rotations in $E$) acting diagonally on $Y$. 
Let $\overline{F}\from S/SO(d) \to S/SO(d)$ be the map 
induced on the quotient space  by
\ref{propo:inducedF}.
Assume that all masses are positive, $U>0$ and \ref{eq:propertyFF} holds. 
Then, if 
$\pi\from S\to S/SO(d)$ 
denotes 
the projection on the quotient,
\[
\pi(\Fix(F))= \Fix(\overline{F})~.
\]
\end{theo}
\begin{proof}
If $\vq\in \Fix(F) \subset S$ is fixed by $F$, then 
$\overline{F} ( \pi(\vq) ) = \pi ( F(\vq) ) = \pi(\vq) \implies
\pi(\vq) \in \Fix(\overline{F})$,
hence $\pi( \Fix(F)) \subset \Fix(\overline{F})$.  
On the other hand,
let $\pi(\vq) \in \Fix(\overline{F})$. 
Then there exists a rotation $g\in SO(E)$ such that $F(\vq) = g \vq$.
Without loss of generality, we can assume that 
$g= e^\Omega$, 
where $\Omega$ is an antisymmetric $d\times d$ matrix,
with $k$ $(2\times2)$-blocks on the diagonal
$\begin{bmatrix}0&-\theta_i\\\theta_i&0\end{bmatrix}$,
with $\theta_{i} \in [-\pi,\pi]$ for each $i=1,\ldots, k$,
and, if $d$ is odd, a one-dimensional diagonal zero entry
($d = 2k$ or $d=2k+1$). We can also assume that only the first 
(say, $l\leq k$) blocks have $\theta_i\neq 0$, 
hence $\Omega$ has $l$ non-zero $(2\times2)$ diagonal blocks
and is zero outside. Note that for each $\vx\in E$, 
the quadratic form $(e^\Omega \vx ) \cdot (\Omega \vx)$ on $E$ 
can be written with $l$ non-singular positive-defined blocks
\[
 \theta_i 
\begin{bmatrix}
\sin \theta_i  &  - \cos \theta_i \\
\cos \theta_i & \sin \theta_i 
\end{bmatrix} 
\sim 
 \theta_i 
\begin{bmatrix}
\sin \theta_i  &  0  \\
0  & \sin \theta_i 
\end{bmatrix} 
\]
on the diagonal, and hence it is non-negative. 
Here the symbol ``$\sim$'' means that the two matrices yield 
the same quadratic form. 
Moreover, if one writes 
$\vx \in E\cong \RR^d$  as $(\vz_1,\vz_2,\ldots, \vz_l,x_{2l+1}, \ldots, x_d)$,
with $\vz_i \in \RR^2$ for $i=1,\ldots, l$ and 
$x_i \in \RR$, 
then 
\begin{equation}\label{eq:thetasintheta}
(e^\Omega \vx ) \cdot (\Omega \vx) = 
\sum_{i=1}^l 
\theta_i \sin \theta_i \norm{\vz_i}^2   ~.
\end{equation} 
Let $p_i \from E \to \RR^2$ denote the projection 
$\vx \mapsto \vz_i$, 
for $i=1,\ldots, l$.  

Since $F(e^{t\Omega}\vq)$ does not depend on $t\in \RR$, 
\[
\begin{aligned}
0 & =\dfrac{d}{dt} \lVert F( e^{t\Omega}\vq ) \rVert^2_M |_{t=0} \\
&= 2 \langle F(\vq), \Omega \vq \rangle_M \\
& = 
2 \langle e^{\Omega} \vq, \Omega \vq \rangle_M   \\
&  = 2 \sum_{j=1}^n  m_j (e^{\Omega} \vq_j ) \cdot (\Omega \vq_j)
\end{aligned}
\]
For each $j=1,\ldots, n$ the inequality $m_j>0$ holds, and 
for each $\vx \in E$ the inequality 
 $( e^{\Omega} \vx ) \cdot (\Omega \vx)\geq 0$ holds:
it follows that  for each $j$,
\(
(e^{\Omega}\vq_j) \cdot (\Omega \vq_j) = 0  \).
By \eqref{eq:thetasintheta},
this implies that,
given $j$, for each $i=1,\ldots, l$
 either $p_i(\vq_j) = 0$ or 
$\theta_i \in \{\pi,-\pi\}$  (since $\theta_i \neq 0$ for $i=1,\ldots, l$).  
If $p_i(\vq_j) = 0$ for each $j$, then actually $g\vq = \vq$, and therefore 
$\pi(\vq) \in \pi ( \Fix(F))$. 
So, without loss of generality one can assume that for each $i=1,\ldots,l$ there 
exists $j$ such that  $p_i(\vq_j) \neq 0$. 
Suppose that $l\geq 1$, and therefore 
$\theta_1 = \pm \pi$. 
By property \ref{eq:propertyFF} 
there exists $\bar j$ such that 
$p_1(F_{\bar j}(\vq)) \cdot p_1(\vq_{\bar j}) \geq 0$ 
and $p_1(\vq_{\bar j}) \neq 0$. 
But this would imply that  
\[ F(\vq) = g\vq = e^{\Omega} \vq \implies 
- p_1(\vq_{\bar j}) \cdot p_1(\vq_{\bar j})  \geq 0 
\implies p_1(\vq_{\bar j}) = 0~,
\]
which is not possible. 
Therefore, if condition \eqref{eq:propertyFF} holds, 
$l=0$, and $g\vq = \vq$. 
The conclusion follows. 
\end{proof}

\section{Projective  fixed points and Morse indices}
\label{sec:projectivemorse}

In this section, we finally prove the equation relating fixed point and 
Morse indices of central configurations.

\begin{propo}
\label{propo:IFD2}
If $U$ is homogeneous of degree $-\alpha$, 
then for each central configuration $\vq$, up to a linear change of coordinates
\[
-\alpha U(\vq) (I - F'(\vq)) =  D^2\tilde U (\vq),
\]
where $F\from S \to S$ is the function of \ref{propo:mapFStoS}, 
defined as $F(\vq) = -\dfrac{\nabla_M U(\vq)}{\norm{\nabla_M U(\vq)}_M }$, 
and $D^2\tilde U(\vq)$ is the Hessian of the restriction $\tilde U$ of $U$ to $S$,
evaluated at $\vq$.  
\end{propo}
\newcommand{\mU}{\mathcal{U}}
\begin{proof}
As above, let 
$X = E^n \smallsetminus \Delta$ denote the
configuration space of $E$.
After a linear change of coordinates in $X$, we can assume $m_i=1$ for each $i$
and $\vq=(1,0,\ldots, 0)=(1,\zero)$ (rescale $\vq$ by a diagonal matrix with 
suitable entries
on its diagonal, and apply a rotation - this leaves $U$ homogeneous of the same degree). 
Given suitable linear coordinates $\vx=(x_0,x_1,\ldots, x_l)$ in  $Y\cong \RR^{l+1}$, 
the ellipsoid  $S$ has equation $\norm{\vx}^2 = 1$, 
and $F(\vx) = - \dfrac{\nabla U}{ \norm{\nabla U} } $,
in the $\vx$-coordinates. 
Therefore, if $\vu= (u_1,\ldots, u_l) \mapsto (\sqrt{1-\norm{\vu}^2}, u_1,\ldots, u_l)\in S$
is a local chart around the central configuration $\vq \sim (1,\zero) $,
\begin{equation}\label{eq:ps}
\begin{aligned}
\dfrac{\partial F_\alpha}{\partial u_\beta}(\zero) &= 
\dfrac{\partial^2 U}{\partial x_\alpha \partial x_\beta}  (\vq)
\norm{ \nabla U(\vq) }^{-1}  \\
D^2_{\alpha\beta} \tilde U(\zero) &= 
\dfrac{\partial^2 U}{\partial x_\alpha \partial x_\beta}(\vq) - \delta_{\alpha\beta} \dfrac{\partial U}{\partial x_0} (\vq) ~.
\end{aligned}
\end{equation}
Now,
$\dfrac{\partial U}{\partial x_0} = \langle \nabla U, \vq\rangle = -\alpha U( \vq)$,
and since $\vq$ is a central configuration, it is a fixed point of $F$ and therefore
\[
F(\vq) = \vq \implies 
\dfrac{\partial U}{\partial x_\alpha} = 0,
\quad
\text{ for } \alpha=1,\ldots, l
\]
and 
\begin{equation}
\label{eq:nablaUpartial}
\dfrac{\partial U}{\partial x_0}(\vq) = - \norm{\nabla U(\vq)} = -\alpha U(\vq) ~.
\end{equation}
From \eqref{eq:nablaUpartial} and 
\eqref{eq:ps} it follows that for $\alpha,\beta=1,\ldots, l$,
\begin{equation}
\begin{aligned}
D^2_{\alpha\beta} \tilde U(\zero) 
+ \delta_{\alpha\beta} \dfrac{\partial U}{\partial x_0}(\vq) 
& = 
\dfrac{\partial^2 U}{\partial x_\alpha \partial x_\beta}(\vq) 
= \dfrac{\partial F_\alpha}{\partial u_\beta}(\zero) 
\norm{ \nabla U(\vq) } \\
\implies 
D^2 \tilde U(\zero) &= - \alpha U(\vq) I   + \alpha U(\vq) F'(\vq) 
\end{aligned}
\end{equation}
\end{proof}

The \term{maximal isotropy stratum} in $S/SO(d)$ 
is the set of all orbits of points in $S$ with trivial isotropy (and hence
with maximal isotropy type). 
It is an open (and dense) subspace of $S/SO(d)$.

\begin{coro}
\label{coro:index}
Assume the hypotheses of Theorem \ref{theo:FF} hold. Then 
for each non-degenerate projective class of central configurations $\vq \in \Fix(\bar F)$ 
in the maximal isotropy stratum of $\bar S=S/SO(d)$, 
the fixed point index $\ind(\vq, \bar F)$ and the 
Morse index $\mu_{\tilde U}(\vq)$ 
of $\tilde U$ at $\vq$ are related by 
the equality 
\[
\ind(\vq, \bar F) = (-1)^{\mu_{\tilde U}(\vq)+\epsilon},
\]
where $\epsilon = 
d(n-1) -1 - d(d-1)/2$.
\end{coro}
\begin{proof}
Since $\tilde U$ is $SO(d)$-invariant (with diagonal action), 
and the $SO(d)$-orbits in $S$ with maximal isotropy type have dimension $d(d-1)/2$, 
the Hessian $D^2\tilde U$ has $d(d-1)/2$-dimensional kernel,
if $\vq$ is non-degenerate. 
By proposition \ref{propo:IFD2}, 
if follows that 
$F'$ has a $d(d-1)/2$-multiple eigenvalue $1$, 
and $\ind(\vq, \bar F)=(-1)^c$,
where $c$ is the number of negative eigenvalues of $-D^2\tilde U$,
i.e. the number of positive eigenvalues of $D^2\tilde U$,
which is equal to 
\( 
\dim(S)-d(d-1)/2 - \mu_{\tilde U}(\vq) \). 
Therefore 
\[
\ind(\vq,\bar F ) = (-1)^ {\dim(S)-d(d-1)/2 - \mu_{\tilde U}(\vq) } 
= (-1) ^ {\epsilon- \mu_{\tilde U}(\vq) }  = (-1) ^ {\mu_{\tilde U}(\vq) + \epsilon } 
\]
where $\epsilon=\dim(S) -d(d-1)/2 = d(n-1) - 1 - d(d-1)/2$.
\end{proof}


\begin{coro}
\label{coro:PP}
Let $U$ be the Newton potential in \eqref{eq:newtonianU}, 
with positive masses, $\alpha>0$, and $\dim(E)=2$. 
Then $X\approx \CC^n\smallsetminus \Delta$, 
and $S/SO(2) \approx \PP^{n-2}(\CC)_0 \subset \PP^{n-2}(\CC)$,
where $\PP^{n-2}(\CC)_0$ is the subset 
of $\PP^{n-1}(\CC)$ defined in projective coordinates as 
$\PP^{n-2}(\CC)_0 \cong \{ [z_1:\ldots, z_n] \in \PP^n : 
\sum_j m_j z_j = 0, \forall i,j, z_i \neq z_j \}$. 
Then for each non-degenerate projective class
of central configurations 
$q\in \Fix(\overline{F})$, with 
 $\overline{F} \from \PP^{n-2}_0(\CC) \to \PP^{n-2}(\CC)$,
\[
\ind(\vq,\overline{F}) = (-1)^{\mu_{\overline{U}}(\vq)}~.
\]
\end{coro}
\begin{proof}
If $d=2$, then $d(n-1)-1-d(d-1)/2 = 2(n-2)$.
\end{proof}

\section{An example of non-planar relative equilibrium }
\label{sec:examples}

\begin{example}[Non-central and non-planar equilibrium solution]
\label{example:c1c2c3}
In $E\cong \RR^3$, 
let $R_x$, $R_y$ and $R_z$ denote rotations of angle $\pi$ 
around the three coordinate axes. 
Fix three non-zero constants $c_1,c_2,c_3$.
Consider the problem with $6$ bodies in $E$, symmetric with respect 
to the group $K$ with non-trivial elements of $\Sigma_6 \times SO(3)$
\[
((34)(56), R_x), 
((12)(56), R_y ),
((12)(34), R_z)~.
\] 
Assume $m_j=1$, for $j=1,\ldots, 6$, and let 
$U$ be the potential  defined on $E^6$ by 
\[
U(\vq) = \sum_{i<j} \dfrac{1-\gamma_i\gamma_j}{\norm{\vq_i - \vq_j}}
\] 
where 
\[
\gamma_1=\gamma_2 = c_1, \quad
\gamma_3=\gamma_4 = c_2, \quad
\gamma_5=\gamma_6 = c_3~.
\]
Now, $U$ is invariant with respect to $K$, and the vertical cylinder 
is $K$-invariant: it follows from 
\ref{lemma:c1c2c3}
that critical points of the restriction of $U$ to $C^K$ are 
equilibrium configurations. 

In other words, three pairs of bodies 
of unit masses, each pair charged with charge $c_j$,
are constrained each pair to belong to one of the coordinate axes
and to be symmetric with respect to the origin.  

The space $X^K = Y^K$ has dimension $3$, and can be parametrized by 
$(x,y,z)$, where $x$, $y$ and $z$ are (respectively) the coordinates along 
the corresponding axis of particles 1, 3 and 5. 
The generic  configuration $\vq\in X^K$ can be written as 
\[
\begin{aligned}
\vq_1 & =  (x,0,0) &&& \vq_2 & = (-x,0,0) \\
\vq_3 & =  (0,y,0) &&& \vq_4 & = (0,-y,0) \\
\vq_5 & = (0,0,z) &&& \vq_6  & = (0,0,-z),\\
\end{aligned}
\]
and the potential $U$ restricted to $X^K$ in such coordinates is 
\[
U(x,y,z) = 
\dfrac{1-c_1^2}{2\abs{x}} 
+
\dfrac{1-c_2^2}{2\abs{y}} 
+
\dfrac{1-c_3^2}{2\abs{z}} 
+
4 \dfrac{1-c_1c_2}{\sqrt{x^2+y^2}}
+ 
4 \dfrac{1-c_1c_3}{\sqrt{x^2+z^2}}
+ 
4 \dfrac{1-c_2c_3}{\sqrt{y^2+z^2}}~.
\]
The vertical cylinder $C^K\subset X^K$  is 
\[
\begin{aligned}
C^K &  = \{ \vq \in X^K : \langle P\vq,\vq\rangle_M = 2\}
\\
&= \{  (x,y,z) \in X^K : 
2x^2 + 2 y^2 = 2 \} ~.
\end{aligned}
\]
Hence an equilibrium solution is a critical point of 
$U(\cos t, \sin t, z)$ with positive value $U>0$.  
Now, assume 
\begin{equation}\label{eq:inequal0}
c_1>1, \quad c_2<-1, \quad  c_3<-1~.
\end{equation} 
Then 
$1-c_1^2 <0$, $1-c_2^2 <0$, $1-c_3^2<0$, 
$1-c_2c_3<0$; moreover, 
$1-c_1c_2>0$ and  $1-c_1c_3>0$.
The restricted potential $U$ is defined as 
\[
U(t,z) = 
\dfrac{1-c_1^2}{2\abs{\cos t}} 
+
\dfrac{1-c_2^2}{2\abs{\sin t}} 
+
\dfrac{1-c_3^2}{2z} 
+
4 (1-c_1c_2) 
+ 
4 \dfrac{1-c_1c_3}{\sqrt{\cos^2 t+z^2}}
+ 
4 \dfrac{1-c_2c_3}{\sqrt{\sin^2 t+z^2}},
\]
and is 
defined on the strip 
$(t,z) \in T= (0,\pi/2)\times (0,+\infty)$. 

If 
\begin{equation}
\label{eq:inequal1}
1-c_1^2 + 8 (1-c_1c_3)<0,
\end{equation}
then for each $(t,z)$
\[
\begin{aligned}
\frac{1-c_1^2}{2\abs{\cos t}} + 4 \frac{1-c_1c_3}{\sqrt{\cos^2 t + z^2}} &
 < 
\frac{1-c_1^2}{2\abs{\cos t}} + 4 \frac{1-c_1c_3}{\abs{\cos t} }  \\
& = \frac{1}{2\abs{\cos{t}}}
\left(
1 - c_1^2 + 8 (1-c_1c_3) 
\right) < 0 
\end{aligned}
\]
and hence  $U(t,z)< 4(1-c_1c_2)$ for each $(t,z)\in T$ 
and $U\to -\infty$ on the boundary of $T$. 

Furthermore, for each $t\in (0,\pi/2)$, 
\[
\begin{aligned}
\dfrac{\partial U}{\partial z} & =  
\dfrac{c_3^2-1}{2z^2} 
+
4z\dfrac{c_2c_3-1}{(z^2+\sin^2t)^{3/2} } 
- 
4z \dfrac{1-c_1c_3}{(z^2+\cos^2 t)^{3/2}}  
\\
&  < 
\dfrac{c_3^2-1}{2z^2} 
+
4\dfrac{c_2c_3-1}{z^2} 
- 
4 z \dfrac{1-c_1c_3}{(z^2+1)^{3/2}}~.
\end{aligned}
\]
It follows that if 
\begin{equation}\label{eq:inequal2}
c_3^2 - 1 + 8 (c_2c_3-1) < 8 ( 1 - c_1c_3),
\end{equation}
then for every $t\in (0,\pi/2)$ there  exists 
$z_0$ such that $z>z_0 \implies \frac{\partial U}{\partial z}(t,z) < 0$.

Thus, whenever both \ref{eq:inequal1} and \ref{eq:inequal2} hold, 
the restriction $U(t,z)$ attains its maximum in the interior of the strip $T$. 
Such a maximum $(t_m,z_m)$ corresponds to an equilibrium configuration if and only if 
$U(t_m,z_m)>0$, from \ref{lemma:c1c2c3}. 
Note that if $t=\pi/4$ and $z=1/2$ then 
\[
U(\pi/4,z) = 
\frac{\sqrt{2}}{2} (2-c_1^2-c_2^2) + 4(1-c_1c_2)
+ 1-c_3^2 + \frac{4(2-c_3(c_1+c_2))}{\sqrt{3/4}},
\] 
which is positive, for example, if $c_1=20$ and $c_2=c_3=-2$. 
Such coefficients satisfy \eqref{eq:inequal0},
\eqref{eq:inequal1} and \eqref{eq:inequal2}, and therefore 
for such a choice of $c_i$ 
there exist non-planar relative equilibrium configurations.

\end{example}



\begin{thebibliography}{10}

\bibitem{MR1320359}
{\sc Albouy, A.}
\newblock Symétrie des configurations centrales de quatre corps.
\newblock {\em C. R. Acad. Sci. Paris Sér. I Math. 320}, 2 (1995), 217--220.

\bibitem{ac98}
{\sc Albouy, A., and Chenciner, A.}
\newblock Le probl\`eme des $n$ corps et les distances mutuelles.
\newblock {\em Invent. Math. 131}, 1 (1998), 151--184.

\bibitem{MR2925390}
{\sc Albouy, A., and Kaloshin, V.}
\newblock Finiteness of central configurations of five bodies in the plane.
\newblock {\em Ann. of Math. (2) 176}, 1 (2012), 535--588.

\bibitem{betti1877sopra}
{\sc Betti, E.}
\newblock Sopra il moto di un sistema di un numero qualunque di punti che si
  attraggono o si respingono tra loro.
\newblock {\em Annali di Matematica Pura ed Applicata (1867-1897) 8}, 1 (1877),
  301--311.

\bibitem{MR2372989}
{\sc Ferrario, D.~L.}
\newblock Planar central configurations as fixed points.
\newblock {\em J. Fixed Point Theory Appl. 2}, 2 (2007), 277--291.

\bibitem{HaM2207019}
{\sc Hampton, M., and Moeckel, R.}
\newblock Finiteness of relative equilibria of the four-body problem.
\newblock {\em Invent. Math. 163}, 2 (2006), 289--312.

\bibitem{Mo90}
{\sc Moeckel, R.}
\newblock On central configurations.
\newblock {\em Math. Z. 205}, 4 (1990), 499--517.

\bibitem{MR856309}
{\sc Pacella, F.}
\newblock Central configurations of the {$N$}-body problem via equivariant
  {M}orse theory.
\newblock {\em Arch. Rational Mech. Anal. 97}, 1 (1987), 59--74.

\bibitem{palais}
{\sc Palais, R.~S.}
\newblock The principle of symmetric criticality.
\newblock {\em Comm. Math. Phys. 69}, 1 (1979), 19--30.

\bibitem{MR0321389}
{\sc Palmore, J.~I.}
\newblock Classifying relative equilibria. {I}.
\newblock {\em Bull. Amer. Math. Soc. 79\/} (1973), 904--908.

\bibitem{MR0363076}
{\sc Palmore, J.~I.}
\newblock Classifying relative equilibria. {II}.
\newblock {\em Bull. Amer. Math. Soc. 81\/} (1975), 489--491.

\bibitem{MR0413647}
{\sc Palmore, J.~I.}
\newblock Classifying relative equilibria. {III}.
\newblock {\em Lett. Math. Phys. 1}, 1 (1975/76), 71--73.

\bibitem{MR1669486}
{\sc Roberts, G.~E.}
\newblock A continuum of relative equilibria in the five-body problem.
\newblock {\em Phys. D 127}, 3-4 (1999), 141--145.

\bibitem{MR2139425}
{\sc Saari, D.~G.}
\newblock {\em Collisions, rings, and other {N}ewtonian {$N$}-body problems},
  vol.~104 of {\em CBMS Regional Conference Series in Mathematics}.
\newblock Published for the Conference Board of the Mathematical Sciences,
  Washington, DC; by the American Mathematical Society, Providence, RI, 2005.

\bibitem{MR0278700}
{\sc Shub, M.}
\newblock Appendix to {S}male's paper: ``{D}iagonals and relative equilibria''.
\newblock In {\em Manifolds -- {A}msterdam 1970 ({P}roc. {N}uffic {S}ummer
  {S}chool)}, Lecture Notes in Mathematics, Vol. 197. Springer, Berlin, 1971,
  pp.~199--201.

\bibitem{MR0309153}
{\sc Smale, S.}
\newblock Topology and mechanics. {I}.
\newblock {\em Invent. Math. 10\/} (1970), 305--331.

\bibitem{MR0321138}
{\sc Smale, S.}
\newblock Topology and mechanics. {II}. {T}he planar {$n$}-body problem.
\newblock {\em Invent. Math. 11\/} (1970), 45--64.

\bibitem{MR532830}
{\sc Spivak, M.}
\newblock {\em A comprehensive introduction to differential geometry. {V}ol.
  {I}}, second~ed.
\newblock Publish or Perish, Inc., Wilmington, Del., 1979.

\bibitem{wintner}
{\sc Wintner, A.}
\newblock {\em The {A}nalytical {F}oundations of {C}elestial {M}echanics}.
\newblock Princeton Mathematical Series, v. 5. Princeton University Press,
  Princeton, N. J., 1941.

\bibitem{Xia}
{\sc Xia, Z.}
\newblock Central configurations with many small masses.
\newblock {\em J. Differential Equations 91}, 1 (1991), 168--179.

\end{thebibliography}

\end{document}